\documentclass[review, 3p]{elsarticle}
\usepackage{epsfig}
\usepackage{amsbsy,latexsym}
\usepackage{amsmath}
\usepackage{amsthm}
\usepackage{amssymb, mathrsfs}
\usepackage[mathscr]{eucal}
\usepackage{hyperref}
\usepackage{color}

\usepackage[normalem]{ulem}
\usepackage{cancel}

\usepackage{ytableau}

\def\<{\langle}
\def\>{\rangle}

\newcommand{\coq}[3]{q_{#2 \, #3}^{#1} }

\newtheorem{prop}{Proposition}

\begin{document}

\title{On some new hook-content identities}
 \author[sas,muni]{Michal Sedl\'ak}\ead{fyzimsed@savba.sk}
   \address[sas]{Institute of Physics, Slovak Academy of Sciences, D\'ubravsk\'a cesta 9, 845 11 Bratislava, Slovakia}
   \address[muni]{Faculty of Informatics,~Masaryk University,~Botanick\'a 68a,~60200 Brno,~Czech Republic}
   \date{ \today}
\author[QUIT,INFN]{Alessandro Bisio}\ead{alessandro.bisio@unipv.it}
 \address[QUIT]{Dipartimento di Fisica, Universit\`{a} di
   Pavia, via Bassi
   6, 27100 Pavia, Italy}
 \address[INFN]{INFN Sezione di Pavia, via Bassi
   6, 27100 Pavia, Italy}

\begin{abstract}
  Based on the work of A. Vershik\cite{Vershik1992}, we introduce two new combinatorial identities.
  We show how these identities can be used to prove a new hook-content identity.
  The main motivation for deriving this identity was a particular optimization problem in the field of quantum information processing.
\end{abstract}

\maketitle

\section{Introduction}

The representation theory of the symmetric group $S_n$
and of the general linear group $GL(d)$  are related by the so-called
Schur-Weyl duality\cite{schur1901klasse}. This famous theorem proves the decomposition
\begin{align}\label{eq:schurweylduality}
  V^{\otimes n} = \bigoplus_{\lambda \vdash n} V_{\lambda} \otimes
  S_{\lambda} \qquad V= \mathbb{C}^d
\end{align}
for the representation of $GL(d) \times S_n$, where $V_\lambda$ is either
zero or a polynomial irreducible representation of $GL(d)$,
$S_\lambda$ is an irreducible representation of $S_n$ and $\lambda$
runs over the partitions of $n$ and is conveniently represented by
Young diagram.  Both the symmetric group and $GL(d)$ (and especially its compact
subgroups $U(d)$ and $SU(d)$) are of paramount importance in
theoretical physics, especially in quantum mechanics.  For example,
$S_n$ is a fundamental symmetry of systems of identical particles and 
unitary groups represent the set of reversible (finite-dimensional) transformations.
Therefore, it is not surprising
that physics community keeps a steady interest in the
representation theory of these groups and in the Shur-Weyl duality,
from the early work of Weyl \cite{weyl1950theory} up to the most recent
applications in quantum computing and quantum information processing.
For example, Equation~\eqref{eq:schurweylduality} denotes a subsystem
decomposition (induced by the symmetry of the system-environment
interaction) in which one can identify error free subsystems
\cite{PhysRevLett.84.2525,doi:10.1063/1.1824213}, or the relevant
subsystems for quantum estimation \cite{PhysRevLett.93.180503}.  These
are just a couple of examples of a much wider variety of applications
(see e.g. Ref. \cite{RevModPhys.84.711} for a review). In the light of
this discussion it is clear that the dimensions of the irreducible
spaces $V_{\lambda}$ and $S_{\lambda}$, are, more often than not, a
crucial piece of information.  The value of $\dim(S_{\lambda})$ and
$\dim(V_{\lambda})$ are given by the hook length formula
\cite{frame1954hook} and the hook-content formula
\cite{stanley_fomin_1999} respectively.  Those celebrated equations
have a nice combinatorial interpretation and, since their discovery,
they have been generalized (see e.g. \cite{CIOCANFONTANINE20111703}
and references therein) and applied in different fields like algebraic
geometry \cite{floystad2011combinatorial} and probability
\cite{greene1979probabilistic}.  Closely related are also the
Littlewood-Richardson rules\cite{fulton1997young} in the expansion
$S_{\lambda} \otimes S_\mu = \bigoplus_{\nu} S_\nu^{\otimes
  c(\lambda,\mu,\nu)} $ and the branching rules for restricting
$S_{\lambda}$ to $S_{n-1} $ and inducing $S_\lambda$ to $S_{n+1}$
\cite{patera1973tables}.

Our work introduces a new identity,
represented by Eq.~\eqref{eq:SSYTandSYT} in Proposition~\ref{prop:newidentity},
which relates
the dimensions $\dim(S_{\lambda})$,  $\dim(V_{\lambda})$,
$\dim(S_{\lambda^{(j)}})$, and  $\dim(V_{\lambda^{(j)}})$
for any Young diagram $\lambda$ consisting of $n$ boxes and 
diagrams $\lambda^{(j)}$ that can be obtained from $\lambda$ by adding a single box.
The proof of our result relies on a couple of combinatorial
identities, Equations \eqref{eq:hookcont1} and \eqref{eq:hookcont2} in
Proposition~\ref{prop:basic-identities}, which
can be of independent interest. Our approach is modeled after the seminal
work \cite{Vershik1992} of A. Vershik, which provides the essential
tools used in this work.

The main motivation for the presented results originates in the problem we were solving [ref!!!] within the field of quantum information processing.
While trying to derive optimal success probability for a problem with an arbitrary number of uses of a unitary transformation, $n$, and an arbitrary dimension of quantum systems, $d$,
we observed several identities involving $n$ and $d$, which the optimality of the solution required. By reducing the problem even further we arrived at the necessity to prove the hook content identity
~\eqref{eq:hookcontentfinal1}
that forms the core of this paper, and the needed identities for our original problem correspond to our final Proposition \ref{prop:newidentity}.

\section{Basic identities}
\label{sec:combidentities}

Let $\{ a_i \}_{i=1}^{2s} $, $s \geq 1$, be elements of an arbitrary
field and let the following coefficients be well defined:
\begin{align}
  \label{eq:defqjmn}
  \begin{aligned}
  &\coq{j}{m}{n} := \prod_{i = m+1}^{j} \left(1 -
  \frac{a_{2i-1}}{a_{2i-1}+ a_{2i}+
  \dots + a_{2j}} \right)
  \prod_{i = j+1}^{n} \left(1 -
  \frac{a_{2i}}{a_{2j+1}+ a_{2j+2}+
    \dots + a_{2i}} \right)   \\
&\mbox{for } 0 \leq m \leq j \leq n \leq s.
  \end{aligned}
\end{align}
In Ref. \cite{Vershik1992} the following result is proved:
\begin{prop}[Vershik]\label{prop:vershik}
  Let $\coq{j}{m}{n} $ be defined as in Equation \eqref{eq:defqjmn}.
  Then, for any
$ m,n,j$
    , $0 \leq m \leq j \leq n \leq s $ the following identities hold
  \begin{align}
    \label{eq:recursion}
    &     \begin{aligned}
      C_{m, n} \,  \coq{j}{m}{n} &=  \sum_{k = m+1}^j a_{2k} \, \coq{j}{k}{n} +
    \sum_{l = j}^{n-1} a_{2l+1} \, \coq{j}{m}{l}
    \\
C_{m, n} &:= a_{2m+1} + a_{2m+2} + \dots + a_{2n},
    \end{aligned} \\
 &   \sum_{j=m}^{n}  \coq{j}{m}{n}  = 1.
      \label{eq:normalization}
  \end{align}
\end{prop}
By rephrasing this result, we could say that $\coq{j}{m}{n}$
are the solution of the recursion  relation \eqref{eq:recursion}.
Proposition~\ref{prop:vershik} is the main tool
for the proof of the following result.
\begin{prop}\label{prop:basic-identities}
  The following identities hold:
  \begin{align}
    \label{eq:hookcont1}
    \sum_{j = m}^{n}
    \left(
    \sum_{i = m+1}^{j}a_{2i-1} - \sum_{i = j+1}^{n}a_{2i}
    \right) \coq{j}{m}{n} &= 0   \\
    \sum_{j = m}^{n}
    \left(
    \sum_{i = m+1}^{j}a_{2i-1} - \sum_{i = j+1}^{n}a_{2i}
    \right)^2 \coq{j}{m}{n} &= \sum_{i=m+1}^{n} a_{2i}
                              \sum_{l=m+1}^{i}a_{2l-1}
                              \label{eq:hookcont2}
  \end{align}
\end{prop}
\begin{proof}[Proof of Equation~\eqref{eq:hookcont1}]
  The proof is by induction on  $n-m$.
  Let us first consider the case $n-m =1$.
  If we fix an arbitrary $m$  we have $n = m+1$
  and Equation~\eqref{eq:hookcont1} gives:
  \begin{align*}
    -a_{2m+2} \, \coq{m}{m}{m+1} + a_{2m+1} \, \coq{m+1}{m}{m+1}
    =  -a_{2m+2}\frac{a_{2m+1}}{a_{2m+1}+a_{2m+2}} +
    a_{2m+1}\frac{a_{2m+2}}{a_{2m+1}+a_{2m+2}} =0.
  \end{align*}
  Next, we fix arbitrary $m$ and $n$
  with $n-m > 1$
  and let us suppose that the thesis holds for any
  $m',n'$ such that $m'\leq n'$ and
  $n'-m' < n-m$.
  By multiplying Equation~\eqref{eq:hookcont1}
  by $C_{m,n}$ and by using
 the recursion formula~\eqref{eq:recursion} we obtain:
  \begin{align*}
  &C_{m,n}\sum_{j = m}^{n}
    \left(
    \sum_{i = m+1}^{j}a_{2i-1} - \sum_{i = j+1}^{n}a_{2i}
    \right) \coq{j}{m}{n} =   A_{m,n} + B_{m,n} \\
    &A_{m,n} :=
    \sum_{j = m}^{n}
    \left(
    \sum_{i = m+1}^{j}a_{2i-1} - \sum_{i = j+1}^{n}a_{2i}
    \right)
    \sum_{k = m+1}^j a_{2k} \, \coq{j}{k}{n} \\
    &B_{m,n} :=
    \sum_{j = m}^{n}
    \left(
    \sum_{i = m+1}^{j}a_{2i-1} - \sum_{i = j+1}^{n}a_{2i}
    \right)  \sum_{l = j}^{n-1} a_{2l+1} \,  \coq{j}{m}{l}
  \end{align*}
  We start with the coefficient  $A_{m,n}$.
  Since the term with $j=m$ is zero, we have
  \begin{align*}
    A_{m,n} &=
    \sum_{j = m+1}^{n}\sum_{k = m+1}^j
    \left(
    \sum_{i = m+1}^{j}a_{2i-1} - \sum_{i = j+1}^{n}a_{2i}
    \right)
      a_{2k} \, \coq{j}{k}{n} = \sum_{k = m+1}^{n}\sum_{j = k}^n
              \left(
    \sum_{i = m+1}^{j}a_{2i-1} - \sum_{i = j+1}^{n}a_{2i}
    \right)
              a_{2k} \, \coq{j}{k}{n} = \\
& = \sum_{k = m+1}^{n}\sum_{j = k}^n
  \left(
  \sum_{i = m+1}^{k}a_{2i-1} +
  \left(
  \sum_{i = k+1}^{j}a_{2i-1}-
  \sum_{i = j+1}^{n}a_{2i}  \right)
    \right)
  a_{2k} \, \coq{j}{k}{n} = \\
            & = \sum_{k = m+1}^{n} a_{2k}
              \left \{
              \sum_{i = m+1}^{k}a_{2i-1}
              \sum_{j = k}^n \coq{j}{k}{n}
              +
              \sum_{j = k}^n
              \left(
              \sum_{i = k+1}^{j}a_{2i-1}
              - \sum_{i = j+1}^{n}a_{2i}
              \right) \coq{j}{k}{n}
              \right \}= \sum_{k = m+1}^{n} a_{2k}
              \sum_{i = m+1}^{k}a_{2i-1} .
  \end{align*}
  where we used Equation~\eqref{eq:normalization} and
  the inductive hypothesis in the last equality.
  Now we evaluate the coefficient $B_{m,n}$.
  Since the term with $j = n$ is zero we have
  \begin{align*}
    B_{m,n} &:=
    \sum_{j = m}^{n-1}\sum_{l = j}^{n-1}
    \left(
    \sum_{i = m+1}^{j}a_{2i-1} - \sum_{i = j+1}^{n}a_{2i}
              \right)  a_{2l+1}  \coq{j}{m}{l}
              =\sum_{l = m}^{n-1}\sum_{j = m}^{l}
    \left(
    \sum_{i = m+1}^{j}a_{2i-1} - \sum_{i = j+1}^{n}a_{2i}
              \right)  a_{2l+1}\,  \coq{j}{m}{l} =\\
    &=\sum_{l = m}^{n-1}\sum_{j = m}^{l}
    \left(
    \left(\sum_{i = m+1}^{j}a_{2i-1} - \sum_{i = j+1}^{l}a_{2i}\right)
      -
      \sum_{i = l+1}^{n}a_{2i}
      \right)  a_{2l+1}\,  \coq{j}{m}{l} =\\
    &=\sum_{l = m}^{n-1}  a_{2l+1}
      \left \{
      \sum_{j = m}^{l}
      \left(\sum_{i = m+1}^{j}a_{2i-1} - \sum_{i = j+1}^{l}a_{2i}\right)  \coq{j}{m}{l}
      -
      \sum_{j = m}^{l} \coq{j}{m}{l} \sum_{i = l+1}^{n}a_{2i}
\right \} = - \sum_{l = m}^{n-1}  a_{2l+1}\sum_{i = l+1}^{n}a_{2i}
  \end{align*}
  Finally, we have
  \begin{align*}
    A_{m,n} + B_{m,n} &=
    \sum_{k = m+1}^{n} \sum_{i = m+1}^{k}  a_{2i-1} a_{2k} -
    \sum_{l = m}^{n-1}\sum_{i = l+1}^{n}   a_{2l+1} a_{2i}   = \\
     &= \sum_{k = m+1}^{n}\sum_{i = m+1}^{k} a_{2i-1} a_{2k} -
       \sum_{l' = m+1}^{n} \sum_{i = l'}^{n}   a_{2l'-1} a_{2i}   = \\
     &= \sum_{k = m+1}^{n}\sum_{i = m+1}^{k}  a_{2i-1}  a_{2k}  -
       \sum_{i = m+1}^{n} \sum_{l' = m+1}^{i}  a_{2l'-1}  a_{2i}   = 0
  \end{align*}
  where we defined $l' = l+1$.
\end{proof}
\begin{proof}[Proof of Equation~\eqref{eq:hookcont2}]
  The proof is again by induction and is very similar to the proof of
  Equation~\eqref{eq:hookcont1}.
  Let us first fix an arbitrary $m$
  and 
  $n = m+1$.
  Then, Equation~\eqref{eq:hookcont2} becomes:
  \begin{align*}
      a^2_{2m+2} \, \coq{m}{m}{m+1} + a^2_{2m+1} \, \coq{m+1}{m}{m+1}
    &=  a^2_{2m+2}\frac{a_{2m+1}}{a_{2m+1}+a_{2m+2}} +
    a^2_{2m+1}\frac{a_{2m+2}}{a_{2m+1}+a_{2m+2}} =  a_{2m+2}\, a_{2m+1}.
  \end{align*}
  We now fix arbitrary $m$ and $n$
  with $n-m > 1$
  and let us suppose that the thesis holds for any
  $m',n'$ such that $m'\leq n'$ and
  $n'-m' < n-m$.
  By multiplying the left hand side of
  Equation~\eqref{eq:hookcont2} by $C_{m,n}$ and by inserting   the
  recursion formula~\eqref{eq:recursion}
  we obtain:
  \begin{align*}
& C_{m,n}\sum_{j = m}^{n}
    \left(
    \sum_{i = m+1}^{j}a_{2i-1} - \sum_{i = j+1}^{n}a_{2i}
    \right)^2 \coq{j}{m}{n}   = \alpha_{m,n} + \beta_{m,n} \\
  &  \alpha_{m,n}:=
\sum_{j = m}^{n}
    \left(
    \sum_{i = m+1}^{j}a_{2i-1} - \sum_{i = j+1}^{n}a_{2i}
    \right)^2  \sum_{k = m+1}^j a_{2k}  \, \coq{j}{k}{n}
   \\
    &\beta_{m,n}:=
\sum_{j = m}^{n}
      \left(
    \sum_{i = m+1}^{j}a_{2i-1} - \sum_{i = j+1}^{n}a_{2i}
    \right)^2 \sum_{l = j}^{n-1} a_{2l+1} \, \coq{j}{m}{l}
  \end{align*}
Next we rewrite 
coefficient $\alpha_{m,n}$. Since the term with
$j=m$ is zero we have
\begin{align*}
  \alpha_{m,n}=&\sum_{j = m+1}^{n}  \sum_{k = m+1}^j
                \left(
    \sum_{i = m+1}^{j}a_{2i-1} - \sum_{i = j+1}^{n}a_{2i}
                \right)^2 a_{2k}
                \, \coq{j}{k}{n}
               = \\
              =&\sum_{k = m+1}^{n}\sum_{j = k}^n  a_{2k}
                 \left(
                \sum_{i = m+1}^{k}a_{2i-1} +
                \left(
                \sum_{i = k+1}^{j}a_{2i-1}
                -
                \sum_{i = j+1}^{n}a_{2i}
                \right)
                \right)^2  \coq{j}{k}{n}
                = \\
  =&\sum_{k = m+1}^{n}\sum_{j = k}^n  a_{2k}
                 \left(
                \left(\sum_{i = m+1}^{k}a_{2i-1} \right)^2+
                \left(
                \sum_{i = k+1}^{j}a_{2i-1}
                -
                \sum_{i = j+1}^{n}a_{2i}
    \right)^2+  \right. \\
&+
  \left.
                  2  \left(\sum_{i = m+1}^{k}a_{2i-1} \right)
                \left(
                \sum_{i = k+1}^{j}a_{2i-1}
                -
                \sum_{i = j+1}^{n}a_{2i}
                \right)
                \right)  \coq{j}{k}{n}
  = \\
  =&\sum_{k = m+1}^{n}  a_{2k}
                 \left(
                \left(\sum_{i = m+1}^{k}a_{2i-1} \right)^2+
\sum_{i=k+1}^n \sum_{l=k+1}^i a_{2l-1}\, a_{2i}
                \right)
     \end{align*}
where we used the inductive hypothesis,
Equation~\eqref{eq:normalization}
and Equation~\eqref{eq:hookcont1} that was previously proved.
Similarly, we rewrite the coefficent $\beta_{m,n}$. Since the term with
$j=n$ is zero we have
\begin{align*}
  \beta_{m,n}=&\sum_{j = m}^{n-1}\sum_{l = j}^{n-1}
      \left(
    \sum_{i = m+1}^{j}a_{2i-1} - \sum_{i = j+1}^{n}a_{2i}
                \right)^2 a_{2l+1} \, \coq{j}{m}{l} = \\
  =&\sum_{l = m}^{n-1} \sum_{j = m}^{l}
     \left(\left(
    \sum_{i = m+1}^{j}a_{2i-1} - \sum_{i = j+1}^{l}a_{2i}
    \right) -\sum_{i = l+1}^{n}a_{2i} \right)^2 a_{2l+1} \,
     \coq{j}{m}{l} = \\
  =&\sum_{l = m}^{n-1}  a_{2l+1} \sum_{j = m}^{l}
     \left(\left(
    \sum_{i = m+1}^{j}a_{2i-1} - \sum_{i = j+1}^{l}a_{2i}
    \right)^2 -2 \left(\sum_{i = l+1}^{n}a_{2i} \right) \left(
    \sum_{i = m+1}^{j}a_{2i-1} - \sum_{i = j+1}^{l}a_{2i}
     \right)\right.\\
  &+\left.\left(\sum_{i = l+1}^{n}a_{2i} \right)^2 \right)
    \coq{j}{m}{l} =
    \sum_{l = m}^{n-1} a_{2l+1}
    \left(
    \sum_{i=m+1}^l \sum_{k=m+1}^{i} a_{2k-1} \, a_{2i}+
    \left( \sum_{i=l+1}^n a_{2i} \right)^2
    \right),
\end{align*}
where we have used the inductive hypothesis,
Equation~\eqref{eq:normalization}
and Equation~\eqref{eq:hookcont1}.
Let us inspect the product of the right hand side of
Equation~\eqref{eq:hookcont2}
and $C_{m,n}$. We obtain 
\begin{align}
  C_{m,n}  \left(
  \sum_{i=m+1}^{n} \right.  &a_{2i} \,
   \left. \sum_{l=m+1}^{i}a_{2l-1} \right)  =
  \gamma_{m,n}^{(1)} + \gamma_{m,n}^{(2)}, \\
   \gamma_{m,n}^{(1)}&:=\sum_{i=m+1}^{n} \sum_{l=m+1}^{i}
                       \sum_{k=m+1}^n
                       a_{2i}     a_{2l-1} a_{2k-1},
  &\qquad
  \gamma_{m,n}^{(2)}&:=
                      \sum_{i=m+1}^{n}  \sum_{l=m+1}^{i}
                      \sum_{k=m+1}^n
                      a_{2i}a_{2l-1} a_{2k}.
\end{align}
Let us 
define
  \begin{align*}
   \alpha_{m,n}^{(1)}&:=
  \sum_{k = m+1}^{n}  \sum_{i = m+1}^{k}  \sum_{j = m+1}^{k}
                       a_{2i-1} \, a_{2j-1} \, a_{2k},
    &
      \beta_{m,n}^{(1)}&:=
    \sum_{l=m}^{n-1} \sum_{i=m+1}^l \sum_{k=m+1}^{i}
                       a_{2k-1} a_{2i} a_{2l+1},&
  \\
                                                  \alpha_{m,n}^{(2)}&:=
  \sum_{k = m+1}^{n}
                                                                      \sum_{i=k+1}^n \sum_{l=k+1}^i a_{2l-1}\, a_{2i} \, a_{2k},
                                                                      &
     \beta_{m,n}^{(2)}&:=
    \sum_{l=m}^{n-1} \sum_{i=l+1}^n \sum_{j=l+1}^{n}
                        a_{2i} a_{2j} a_{2l+1} \\
   \alpha_{m,n} &= \alpha_{m,n}^{(1)}+\alpha_{m,n}^{(2)}&\beta_{m,n} &= \beta_{m,n}^{(1)}+\beta_{m,n}^{(2)} .
  \end{align*}
  We then have
  \begin{align*}
    \gamma_{m,n}^{(1)}-\beta_{m,n}^{(1)}=&
    \sum_{l=m}^{n-1}
    \sum_{i=m+1}^{n}
    \sum_{k=m+1}^i
    a_{2l+1} a_{2i} a_{2k-1} -
    \sum_{l=m}^{n-1}
    \sum_{i=m+1}^l
    \sum_{k=m+1}^{i}
    a_{2k-1} a_{2i} a_{2l+1} = \\
    =&
    \sum_{l=m}^{n-1}
    \sum_{i=l+1}^n
    \sum_{k=m+1}^{i}
    a_{2k-1} a_{2i} a_{2l+1} =
  \sum_{l=m+1}^{n}
    \sum_{i=l}^n
    \sum_{k=m+1}^{i}
       a_{2k-1} a_{2i} a_{2l-1} = \\
    =&\sum_{i=m+1}^n
    \sum_{l=m+1}^{i}
    \sum_{k=m+1}^{i}
       a_{2k-1} a_{2i} a_{2l-1} = \alpha_{m,n}^{(1)}
  \end{align*}
  In a similar way we have
  \begin{align*}
    \alpha_{m,n}^{(2)}+\beta_{m,n}^{(2)} &=
    \sum_{j = m+1}^n
    \sum_{i=j+1}^{n}
    \sum_{l=j+1}^i
    a_{2j}
    a_{2i}
    a_{2l-1} +
    \sum_{l=m+1}^{n}
    \sum_{i=l}^{n}
    \sum_{j=l}^{n}
    a_{2l-1}
    a_{2i}
    a_{2j}=\\
    &=\sum_{j=m+1}^{n}
    \sum_{l=j+1}^{n}
    \sum_{i=l}^{n}
    a_{2j}
    a_{2l-1}
    a_{2i}
    +
    \sum_{i=m+1}^{n}
    \sum_{l=m+1}^{i}
    \sum_{j=l}^{n}
    a_{2i}
    a_{2l-1}
    a_{2j}=\\
    &=\sum_{i=m+1}^{n}
    \sum_{l=m+1}^{n}
    \sum_{j=l}^{n}
    a_{2i}
    a_{2l-1}
    a_{2j}=
    \sum_{i=m+1}^{n}
    \sum_{j=m+1}^{n}
    \sum_{l=m+1}^{j}
    a_{2i}
    a_{2l-1}
    a_{2j}=
    \gamma_{m,n}^{(2)} \;.
  \end{align*}
  Therefore
  $\alpha_{m,n}^{(1)}+\alpha_{m,n}^{(2)}+\beta_{m,n}^{(1)}+\beta_{m,n}^{(2)}
  = \gamma_{m,n}^{(1)}+\gamma_{m,n}^{(2)}$ which finally proves the thesis.
\end{proof}

\section{Hook-content identities and $GL(n)$}

For a natural number $n$, we denote with
$\lambda = (\lambda_1, \lambda_2, \dots, \lambda_k)$
$\lambda_i \in \mathbb{Z} $,
$\lambda_1 \geq \lambda_2 \geq \dots \geq \lambda_k >0 $,
$\sum_{i=1}^{k}\lambda_i = n $ a \emph{partition} of $n$ and we write
$\lambda \vdash n$, Any partition corresponds to a \emph{Young
  diagram}, which is an array of boxes, in the plane, left-justified,
with $\lambda_i$ cells in the $i$-th row from the top (English convention).
A greek letter $\lambda$ denotes both the partition and the
corresponding Young diagram.  A box $b \in \lambda$ of a Young diagram
can be labeled by a pair of integer numbers $b=(i,j)$ where $i$
denotes the row and $j$ denotes the column.  We denote with $h(b)$ the
\emph{hook lenght} of the box $b = (i,j)$, i.e. the number of boxes
$b'=(i',j')$ such that $i=i'$ and $j' \geq j $ or $j=j'$ and
$i' \geq i$. For example, if $\lambda = (4,3,1)$
and $b = (1,2)$ we have $h(b) = 4$.
The \emph{content} of the box $b=(i,j)$ is defined as
$c(b) := j - i $.  A \emph{Young tableau} of shape $\lambda$ is a Young
diagram $\lambda$ in which each box is filled with an integer
number. A \emph{semi-standard} Young tableau of parameters
$(d, \lambda)$, is a Young tableau of shape $\lambda$ such that the
entries are positive integers no greater that $d$ and they weakly
increase along rows and strictly increase along columns.
For example the following tableau
\begin{align*}
  \begin{aligned}
  \ytableausetup
  {mathmode, boxsize=1.5em}
\begin{ytableau}
 1 & 2  & 2 & 3 \\
 2 & 3  \\
 4
\end{ytableau}
  \end{aligned}
\;\; ,
\end{align*}
is a semistandard Young tableau
of parameters $(4, (4,2,1) )$.
The Stanley's \emph{hook-content formula}\cite{stanley_fomin_1999}
gives the number of semistandard Young tableaux
of parameters $(\lambda, d)$ (denoted with $ SSYT(\lambda , d)$),  namely
\begin{align}
  \label{eq:hookcontentformula}
  SSYT(\lambda , d) = \prod_{b \in \lambda}\frac{d - c(b) }{h(b)}.
\end{align}
The number $SSYT(\lambda , d)$ is the dimension $\dim (V_\lambda)$ of
the vector space $V_\lambda$ which is
 the irreducible polynomial representation of $GL(d)$  labelled by the partition
$\lambda$.

A \emph{standard} Young  tableau of shape $\lambda \vdash n$
is semistandard Young tableau of parameters $(n, \lambda)$ such that the filling is
a bijective assignment of $1,2...,n$.
For example the following tableau
\begin{align*}
  \begin{aligned}
  \ytableausetup
  {mathmode, boxsize=1.5em}
\begin{ytableau}
 1 & 2  & 5 & 7 \\
 3 & 4  \\
 6
\end{ytableau}
  \end{aligned}
\;\; ,
\end{align*}
is a standard Young tableau of shape $(4,2,1)$.  The number of
standard Young tableaux of shape $\lambda \vdash n$ (denoted with
$SYT(\lambda)$) is given by the \emph{hook length formula}\cite{frame1954hook}
\begin{align}
  \label{eq:hookformula}
  SYT(\lambda) = \prod_{b \in \lambda}\frac{n!}{h(b)}.
\end{align}
The number $SYT(\lambda)$ is the dimension $\dim(S_{\lambda})$
of the Specht module $S_\lambda$, i.e. the irreducible
representation of the symmetric group $S_n$ labelled by the partition
$\lambda$.

We now introduce the \emph{step coordinates} for a Young diagram
$\lambda$. First, let us define the following notation for a partition
$\lambda = ((\lambda'_1 , k'_1), (\lambda'_2 , k'_2), \dots ,
(\lambda'_s , k'_s ) )$, where $k'_i$ denotes the multiplicity of the
number $\lambda'_i$ and
$\lambda'_1 > \lambda'_2 >\dots > \lambda'_s $.  For example we have
$(4,4,4,3,3,1,1,1,) = ((4,3),(3,2),(1,3))$ Then we define
$p_1 = \lambda'_s$, $p_i := \lambda'_{s-i+1} - \lambda'_{s-i+2}$ for
$i =2,\dots s$ and $k_i := k'_{s-i+1}$ for $i=1, \dots, s$.  The
numbers $(p_1,k_1 ,p_2, k_2  , \dots, p_s , k_s)$ are the step
coordinates of the Young diagram $\lambda$, the reason for this
notation is clear by looking at the example in
Figure~\ref{fig:stepcoordinate}.
\begin{figure}[h!]
  \begin{center}
  \includegraphics[width=7cm]{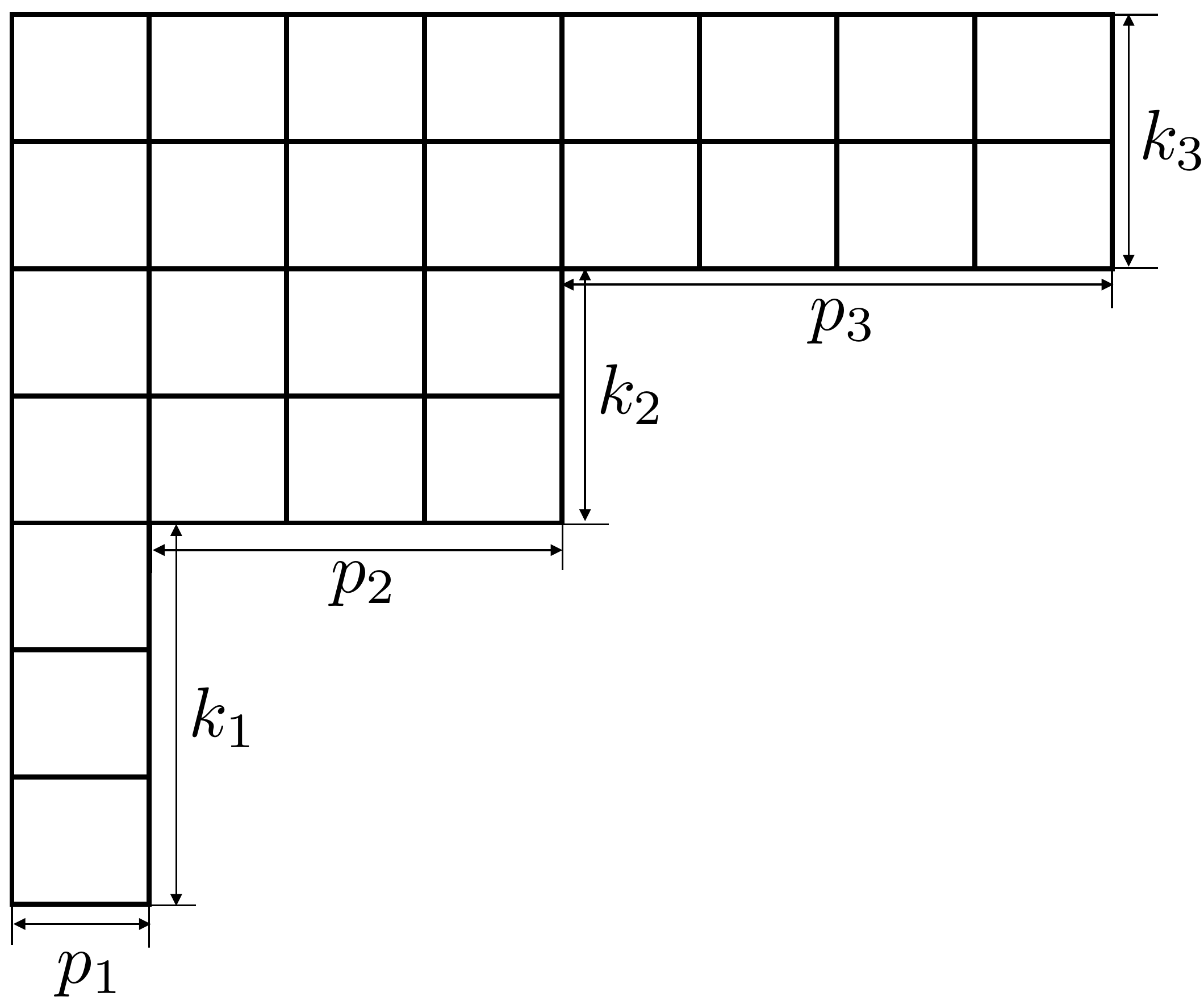}
  \end{center}
  \caption{The Young diagram $\lambda =
    (8,8,4,4,1,1,1)$
  has step coordinates $(1,3, 3,2 ,4,2)$}
  \label{fig:stepcoordinate}
\end{figure}

For a given Young diagram $\lambda$ we denote with $\lambda^{(j)}$
the Young diagram that can be obtained from $\lambda$ by the addition
of the box $b^{(j)}$, for example, if $\lambda = ( 3,3,2,2 )$ we have
\begin{align*}
  \lambda =
 \begin{aligned}
  \ytableausetup
  {mathmode, boxsize=1.8em}
\begin{ytableau}
  \, & \,  & \, \\
  \, & \,  & \, \\
  \, & \, \\
  \, & \,
\end{ytableau}
  \end{aligned}
      \qquad
      \lambda^{(0)} =
 \begin{aligned}
  \ytableausetup
  {mathmode, boxsize=1.8em}
\begin{ytableau}
  \, & \,  & \, \\
  \, & \,  & \, \\
  \, & \, \\
  \, & \, \\
 *(yellow) b^{(0)}
\end{ytableau}
  \end{aligned}
      \qquad
        \lambda^{(1)} =
  \begin{aligned}
  \ytableausetup
  {mathmode, boxsize=1.8em}
\begin{ytableau}
  \, & \,  & \, \\
  \, & \,  & \, \\
  \, & \, & *(yellow) b^{(1)} \\
  \, & \, \\
\end{ytableau}
  \end{aligned}
      \qquad
      \lambda^{(2)} =
    \begin{aligned}
  \ytableausetup
  {mathmode, boxsize=1.8em}
\begin{ytableau}
  \, & \,  & \, & *(yellow) b^{(2)} \\
  \, & \,  & \, \\
  \, & \, \\
  \, & \, \\
\end{ytableau}
  \end{aligned} \quad .
\end{align*}
This is clearly equivalent to say that, for a Young diagram
$\lambda \vdash n$, $\lambda^{(j)}$ is a partition of $n+1$ such that
the Young diagram of $\lambda$ fits inside that of $\lambda^{(j)}$,
and we write $ \lambda^{(j)} \leftarrow \lambda$.

The connection between combinatorial identities studied in Section \ref{sec:combidentities}
and the representation theory is provided by the following observation.
Let us now consider an arbitrary Young diagram $\lambda$
with step coordinates $(p_1,k_1 ,p_2, k_2  , \dots, p_s , k_s)$
and let us apply Equation \eqref{eq:defqjmn}
to the sequence
$a_{2i-1} := p_i$, $a_{2i} := k_i$ $i=1,\dots,  s$.
Then, using the definition of hook-lenght, we have\footnote{This observation appears in the work of Vershik
  \cite{Vershik1992} (see Eq. (17) therein), with
  different notation for the step coordinates.} 
  \begin{align}
    \label{eq:qwithhook}
    \coq{j}{0}{s} =
    \frac
    {\prod_{b\in \lambda} h(b) }
    { \prod_{b\in \lambda^{(j)}} h(b)}  \qquad \mbox{for } j =0, \dots, s,
  \end{align}
  for any Young diagram $\lambda \vdash n $, and $\lambda^{(j)} \vdash
  n+1$,
  $\lambda^{(j)} \leftarrow \lambda $.
Applying this observation to combinatorial identities
~\eqref{eq:hookcont1} and ~\eqref{eq:hookcont2}
and realizing that
  \begin{align}
    \sum_{i = 1}^{j} a_{2i-1} - \sum_{i = j+1}^{s} a_{2i} = \sum_{i = 1}^{j} p_{i} - \sum_{i = j+1}^{s} k_{i} &= c(b^{(j)}) \nonumber \\
     \sum_{i=1}^{s} a_{2i} \sum_{l=1}^{i} a_{2l-1}= \sum_{i=1}^{s} k_i \sum_{l=1}^i p_l &= n \nonumber  \;\; ,
  \end{align}
 we obtain a proof of  the following hook-content identities:
  \begin{align}
    \label{eq:hookcontentfinal1}
    \sum_{\lambda^{(j)} \leftarrow \lambda}
c(b^{(j)})
    \frac
    {\prod_{b\in \lambda} h(b) }
    { \prod_{b\in \lambda^{(j)}} h(b)}  = 0 \\
    \label{eq:hookcontentfinal2}
    \sum_{\lambda^{(j)} \leftarrow \lambda}
(c(b^{(j)}))^2
    \frac
    {\prod_{b\in \lambda} h(b) }
    { \prod_{b\in \lambda^{(j)}} h(b)}  = n
  \end{align}
As consequence of the above Equations \eqref{eq:hookcontentfinal1} and
\eqref{eq:hookcontentfinal2} we also have:
\begin{prop}\label{prop:newidentity}
  For any Young diagram $\lambda \vdash n $,
  and $\lambda^{(j)} \vdash
  n+1$,
   $\lambda^{(j)} \leftarrow \lambda $ we have
   \begin{align}
     \label{eq:SSYTandSYT}
     \sum_{ \lambda^{(j)} \leftarrow \lambda}
     \left ( \frac{SSYT(\lambda^{(j)},d)}{SSYT(\lambda,d)}\right)^2
     \frac{SYT(\lambda)}{SYT(\lambda^{(j)})} = \frac{n+d^2}{n+1}
   \end{align}
\end{prop}
\begin{proof}
  By expanding Equation~\eqref{eq:SSYTandSYT} we obtain
  \begin{align*}
    (n+1) \sum_{ \lambda^{(j)} \leftarrow \lambda}
    & \left(
      \frac{SSYT(\lambda^{(j)},d)}{(SSYT(\lambda,d)}
      \right)^2
    \frac{SYT(\lambda)}{SYT(\lambda^{(j)})} =
   \sum_{ \lambda^{(j)} \leftarrow \lambda}
     (d- c(b^{(j)}))^2  \frac
    {\prod_{b\in \lambda} h(b) }
    { \prod_{b\in \lambda^{(j)}} h(b)}  = \\
     &
= \sum_{ \lambda^{(j)} \leftarrow \lambda}
       \left(
       d^2 -2d c(b^{(j)}) +     (c(b^{(j)}))^2
       \right)
       \frac
    {\prod_{b\in \lambda} h(b) }
    { \prod_{b\in \lambda^{(j)}} h(b)}
      = d^2 + n   \; ,
  \end{align*}
  where the last equality follows from
  Equations~\eqref{eq:normalization},~\eqref{eq:qwithhook},~\eqref{eq:hookcontentfinal1}
  and \eqref{eq:hookcontentfinal2}.
\end{proof}
We conclude this section by noticing that
Equation~\eqref{eq:hookcontentfinal1},
can be alternatively proved by combining the hook content
formula~\eqref{eq:hookcontentformula} and easy consideration of
representation theory.
Indeed, let us consider the trivial identity
\begin{align}
  \label{eq:trivialidentity}
d \dim (V_{\lambda} ) = \sum_{ \lambda^{(j)} \leftarrow \lambda}
\dim(V_{\lambda^{(j)}})
\end{align}
where
$V_\lambda$ an irreducible polynomial
representation of $GL(d)$, $V$ is the fundamental (or defining)
representation of $GL(d)$, and $V_{\lambda^{(j)}}$ are the
irreducible inequivalent polynomial
representations in the decomposition
$ V_{\lambda} \otimes V = \sum_{ \lambda^{(j)} \leftarrow \lambda}
V_{\lambda^{(j)}}$.
From the hook content formula Equation~\eqref{eq:trivialidentity}
becomes
\begin{align*}
  \label{eq:trivialidentity2}
d = \sum_{ \lambda^{(j)} \leftarrow \lambda}
  \frac{\dim(V_{\lambda^{(j)}}) }{ \dim (V_{\lambda} ) } &=
  \sum_{ \lambda^{(j)} \leftarrow \lambda}(d - c(b^{(j)})) \frac
    {\prod_{b\in \lambda} h(b) }
                                                           { \prod_{b\in \lambda^{(j)}} h(b)}
                                                           \implies \sum_{ \lambda^{(j)} \leftarrow \lambda}
  c(b^{(j)}) \frac
    {\prod_{b\in \lambda} h(b) }
  { \prod_{b\in \lambda^{(j)}} h(b)} =0
\end{align*}
where we used Equations ~\eqref{eq:qwithhook} and ~\eqref{eq:normalization} in the final step.

\section{Conclusion}

In this paper we proved some new combinatorial identities,
Equations~\eqref{eq:hookcont1} and \eqref{eq:hookcont2}, which can be
proved following the techniques of Ref.\cite{Vershik1992}.
These identities lead to a couple of hook content identities
(Equations~\eqref{eq:hookcontentfinal1} and
\eqref{eq:hookcontentfinal2}). The first identity can be proved with
easy arguments from representation theory, and our approach provides
an alternative proof. On the other hand,
Equation~\eqref{eq:hookcontentfinal2} and Equation~\eqref{eq:SSYTandSYT}
are new.

The representation theory of the symmetric group and of the general
linear group play a significant role in many areas of quantum
information, as discussed e.g. in Ref. \cite{harrowthesis}.
In particular, Equation~\eqref{eq:SSYTandSYT} is directly linked to
the optimal solution of the perfect probabilistic storing and
retrieving of an unknown unitary
transformation \cite{perfectlearningdraft}, which was the problem that led us to prove the presented hook-content identities.

\section*{Acknowledgement}
A. B. is supported by the John Templeton Foundation under the
project ID\# 60609 Quantum Causal Structures. The opinions expressed in
this publication are those of the author and do not necessarily
reflect the views of the John Templeton Foundation.
M.S. acknowledges the support by the QuantERA project HIPHOP (project ID 731473), projects QETWORK (APVV-14-0878), MAXAP (VEGA 2/0173/17), GRUPIK (MUNI/G/1211/2017) and the support of the Czech Grant Agency (GAČR) project no. GA16-22211S. The authors are grateful to M. Ziman for fruitful discussions and collaborative work, which lead to formulation of identity, which is proved in this manuscript.

\section*{References}

\bibliographystyle{elsarticle-num}
\bibliography{bibliography}

\end{document}